%% file: perfect_IFG-formulas.tex
\documentclass[10pt]{article}
\usepackage{amsmath,amssymb,amsthm,mathrsfs}
\usepackage[all]{xy}

\title{Perfect IFG-formulas}
\author{ALLEN L.~MANN \\
Department of Mathematics \\
Colgate University \\
13 Oak Drive \\
Hamilton, NY 13346 \\
USA \\
E-mail: amann@mail.colgate.edu}

\let	\iso			=	\cong

\let	\given		=	|

\let\abs=\envert

\let\norm=\enVert

\newcommand{\set}[1]{\{#1\}}

\newcommand{\setof}[2]{\{\,#1 \mid #2\,\}}

\newcommand{\tuple}[1]{(#1)}
\newcommand{\seq}[1]{\langle#1\rangle}

\newcommand{\concat}{^\frown}
\let	\tuple	=	\seq

\newcommand{\restrict}{\!\upharpoonright\!}

\DeclareMathOperator{\pr}{pr}

\DeclareMathOperator{\Sub}{Sub}

\DeclareMathOperator{\Cyls}{\mathfrak{Cs}}

\DeclareMathOperator{\Suit}{Suit}
\DeclareMathOperator{\DSuit}{DSuit}

\newcommand{\powerset}{\mathscr P}

\newcommand{\A}{\mathfrak{A}}

\newcommand{\C}{\mathfrak{C}}

\newcommand{\url}[1]{\texttt{#1}} 

\newtheorem{theorem}{Theorem}[section]
\newtheorem{corollary}[theorem]{Corollary}
\newtheorem{lemma}[theorem]{Lemma}
\newtheorem{proposition}[theorem]{Proposition}

\theoremstyle{definition}
\newtheorem*{definition}{Definition}

\theoremstyle{remark}

\newcommand{\propref}[1]{Proposition~\ref{#1}}

\newcommand{\hneg}{\sim\!}
\newcommand{\hand}[1]{\land\!_{/#1}\,}
\newcommand{\hor}[1]{\lor\!\!_{/#1}\,}

\newcommand{\hforall}[2]{\forall {#1}_{/#2}}
\newcommand{\hexists}[2]{\exists {#1}_{/#2}}
\newcommand{\modelt}{\models^+}
\newcommand{\modelf}{\models^-}
\newcommand{\modeltf}{\models^\pm}
\newcommand{\modelft}{\models^\mp}

\newcommand{\trump}[1]{\norm{#1}^+}
\newcommand{\cotrump}[1]{\norm{#1}^-}

\newcommand{\abelard}{Ab\'elard}
\newcommand{\eloise}{Elo\"ise}
\newcommand{\toind}[1]{\underset{#1}{\to}}

\newcommand{\n}[1]{{#1^{\hbox{\sm\char91}}}}

\renewcommand{\cong}{\equiv}
\renewcommand{\cong}{\equiv}

\font\sm = cmsy5

\begin{document}
\maketitle
\begin{abstract}
IFG logic \cite{Dechesne:2005} is a variant of the independence-friendly logic of Hintikka and Sandu \cite{Hintikka:1989, Hintikka:1996}. We answer the question: ``Which IFG-formulas are equivalent to ordinary first-order formulas?'' We use the answer to show that the ordinary cylindric set algebra over a structure can be embedded into a reduct of the IFG-cylindric set algebra over the structure. 
\end{abstract}





\input{1-Introduction}

\input{2-Perfect_IFG_formulas}

\input{3-Embedding_CsA}
\input{4-Conclusion}

	
\bibliographystyle{plain}
\bibliography{math}

\end{document}

%% file: 1-Introduction.tex

\section{Introduction}

IFG-cylindric set algebras were introduced in \cite{Mann:2008, Mann:2007} as a way to study the algebra of IFG logic. We recall the relevant definitions and theorems for the reader's convenience.

\begin{definition}
Given a first-order signature $\sigma$, an \emph{atomic IFG-formula}\index{IFG-formula!atomic|mainidx} is a pair $\tuple{\phi, X}$ where $\phi$ is an atomic first-order formula and $X$ is a finite set of variables that includes every variable that appears in $\phi$ (and possibly more).
\end{definition}

\begin{definition}
Given a first-order signature $\sigma$, the language $\mathscr L_\mathrm{IFG}^\sigma$\index{$\mathscr L_\mathrm{IFG}^\sigma$|mainidx}\index{IFG-formula|(} is the smallest set of formulas such that:
\begin{enumerate}
	\item Every atomic IFG-formula is in $\mathscr L_\mathrm{IFG}^\sigma$.
	\item If $\tuple{\phi, Y}$ is in $\mathscr L_\mathrm{IFG}^\sigma$ and \(Y \subseteq X\), then $\tuple{\phi, X}$ is in $\mathscr L_\mathrm{IFG}^\sigma$.
	\item If $\tuple{\phi, X}$ is in $\mathscr L_\mathrm{IFG}^\sigma$, then $\tuple{\hneg\phi, X}$ is in $\mathscr L_\mathrm{IFG}^\sigma$.
	\item If $\tuple{\phi, X}$ and $\tuple{\psi, X}$ are in $\mathscr L_\mathrm{IFG}^\sigma$, and \(Y \subseteq X\), then $\tuple{\phi \hor{Y} \psi, X}$ is in $\mathscr L_\mathrm{IFG}^\sigma$.
	\item If $\tuple{\phi, X}$ is in $\mathscr L_\mathrm{IFG}^\sigma$, \(x \in X\), and \(Y \subseteq X\), then $\tuple{\hexists{x}{Y}\phi, X}$ is in $\mathscr L_\mathrm{IFG}^\sigma$.
\end{enumerate}
Above $X$ and $Y$ are finite sets of variables.
\end{definition}

From now on we will make certain assumptions about IFG-formulas that will allow us to simplify our notation. First, we will assume that the set of variables of $\mathscr L_\mathrm{IFG}^\sigma$ is $\setof{v_n}{n \in \omega}$. Second, since it does not matter much which particular variables appear in a formula, we will assume that variables with smaller indices are used before variables with larger indices. More precisely, if $\tuple{\phi, X}$ is a formula, \(v_j \in X\), and \(i \leq j\), then \(v_i \in X\). By abuse of notation, if $\tuple{\phi, X}$ is a formula and \(\abs X = N\), then we will say that $\phi$ has $N$ variables and write $\phi$ for $\tuple{\phi, X}$. As a shorthand, we will call $\phi$ an IFG$_N$-formula\index{IFG$_N$-formula}. Let \(\mathscr L^\sigma_{\mathrm{IFG}_N} = \setof{\phi \in \mathscr L^\sigma_{\mathrm{IFG}}}{\phi \text{ has $N$ variables}}\)\index{$\mathscr L^\sigma_{\mathrm{IFG}_N}$}. Third, sometimes we will write $\phi \hor{J} \psi$ instead of $\phi \hor{Y} \psi$ and $\hexists{v_n}{J}\phi$ instead of $\hexists{v_n}{Y}\phi$, where \(J = \setof{j}{v_j \in Y}\). Finally, we will use \(\phi \hand{J} \psi\) to abbreviate \(\hneg(\hneg\phi \hor{J}\!\!\hneg\psi)\) and \(\hforall{v_n}{J} \phi\) to abbreviate \(\hneg\hexists{v_n}{J}(\hneg\phi)\).

\begin{definition} 
Let $\phi$ be an IFG-formula. The \emph{subformula tree} of $\phi$, denoted $\Sub(\phi)$\index{$\Sub(\phi)$ \quad subformula tree of $\phi$|mainidx}, is the smallest tree satisfying the following conditions.
\begin{enumerate}
	\item \(\tuple{\emptyset, \phi} \in \Sub(\phi)\).
	\item If \(\tuple{s, \hneg\psi} \in \Sub(\phi)\), then \(\tuple{s\concat 0, \psi} \in \Sub(\phi)\).
	\item If \(\tuple{s, \psi_1 \hor{J} \psi_2} \in \Sub(\phi)\), then \(\tuple{s\concat 1, \psi_1} \in \Sub(\phi)\) and \(\tuple{s\concat 2, \psi_2} \in \Sub(\phi)\).
	\item If \(\tuple{s, \hexists{v_n}{J}\psi} \in \Sub(\phi)\), then \(\tuple{s\concat 3, \psi} \in \Sub(\phi)\).
\end{enumerate}
For every \(\tuple{s, \psi} \in \Sub(\phi)\), \(\tuple{s, \psi} \in \Sub^+(\phi)\)\index{$\Sub(\phi)$ \quad subformula tree of $\phi$!$\Sub^+(\phi)$ \quad positive subformula tree of $\phi$|mainidx} if $s$ contains an even number of 0s, and \(\tuple{s, \psi} \in \Sub^-(\phi)\)\index{$\Sub(\phi)$ \quad subformula tree of $\phi$!$\Sub^-(\phi)$ \quad negative subformula tree of $\phi$|mainidx} if $s$ contains an odd number of 0s.
\end{definition}

From now on, we will assume that all subformulas are indexed by their position in the subformula tree. This will allow us to distinguish between multiple instances of the same formula that may occur as subformulas of $\phi$. For example, if $\phi$ is \(v_0 = v_1 \hor{v_0} v_0 = v_1\) we will distinguish between the left and right disjuncts. \index{IFG-formula|)}

Truth and falsity for IFG-formulas is defined in terms of semantic games. If $\phi$ is an IFG$_N$-formula and \(V,W \subseteq \^NA\), then \(\A \modelt \phi[V]\) iff \eloise\ has a winning strategy for the corresponding semantic game, assuming she knows the initial valuation belongs to $V$. Dually, \(\A \modelf \phi[W]\) iff \abelard\ has a winning strategy for the game, assuming he knows the initial valuation belongs to $W$. In the first case, we say that $V$ is a \emph{winning team} (or \emph{trump}) for $\phi$ in $\A$, and in the second case, we say that $W$ is a \emph{losing team} (or \emph{cotrump}) for $\phi$ in $\A$. We say that \(\A \modeltf \phi\) iff \(\A \modeltf \phi[\^NA]\). 

The purpose of the slashed subscripts in an IFG-formula is to restrict the information available to the players. For example, the IFG-formula \(\hforall{v_0}{\emptyset} \hexists{v_1}{v_0}(v_0 = v_1)\) is not true in any structure with more than one element because after \abelard\ chooses the value of $v_0$, \eloise\ is forced to choose the value of $v_1$ in ignorance of \abelard's choice. However, \(\hforall{v_0}{\emptyset} \hexists{v_1}{v_0}(v_0 = v_1)\) is not false either because there is always the possibility that \eloise\ will guess correctly.

Wilfrid Hodges made an important breakthrough when he found a way to define a Tarski-style semantics for independence-friendly logic \cite{Hodges:1997a, Hodges:1997b}.

\begin{definition} 
Two valuations \(\vec a, \vec b \in\, \^NA\) \emph{agree outside of \(J \subseteq N\)}, denoted \(\vec a \approx_J \vec b\)\index{\(\vec a \approx_J \vec b\)\quad $\vec a$ and $\vec b$ agree outside of $J$|mainidx}, if 
\[
\vec a \restrict (N\setminus J) = \vec b \restrict(N\setminus J).
\]
\end{definition}

\begin{definition} 
Let \(V \subseteq \^NA\), and let $\mathscr U$ be a cover of $V$. The cover $\mathscr U$ is called \emph{$J$-saturated}\index{cover!$J$-saturated|mainidx} if every \(U \in \mathscr U\) is closed under $\approx_J$. That is, for every \(\vec a, \vec b \in V\), if \(\vec a \approx_J \vec b\) and \(\vec a \in U \in \mathscr U\), then \(\vec b \in U\).
\end{definition}


\begin{definition} 
Define a partial operation $\bigcup_J$\index{$\bigcup_J \mathscr U$|mainidx} on collections of sets of valuations by setting \(\bigcup_J \mathscr U = \bigcup \mathscr U\) whenever $\mathscr U$ is a $J$-saturated disjoint cover of $\bigcup \mathscr U$ and letting $\bigcup_J \mathscr U$ be undefined otherwise. Thus the formula \(V = \bigcup_J \mathscr U\) asserts that $\mathscr U$ is a $J$-saturated disjoint cover of $V$. We will use the notation \(V_1 \cup_J V_2\)\index{$V_1 \cup_J V_2$|mainidx} to abbreviate \(\bigcup_J \set{V_1, V_2}\), the notation \(V_1 \cup_J V_2 \cup_J V_3\) to abbreviate \(\bigcup_J \set{V_1, V_2, V_3}\), et cetera.
\end{definition}

\begin{definition} 
A function \(f\colon V \to A\) is \emph{independent of $J$}, denoted \(f\colon V \toind{J} A\)\index{$f\colon V \toind{J} A$ \quad $f$ is independent of $J$|mainidx},  if \(f(\vec{a}) = f(\vec{b})\) whenever \(\vec{a} \approx_J \vec{b}\). 
\end{definition}

\begin{definition} 
If \(\vec a \in\, \^NA\), \(b \in A\), and \(n < N\), define $\vec a(n:b)$\index{$\vec a(n:b)$|mainidx} to be the valuation that is like $\vec a$ except that $v_n$ is assigned the value $b$ instead of $a_n$. In other words,
\[
\vec a(n:b) = \vec a\restrict(N\setminus \set{n}) \cup \set{\tuple{n,b}}.
\]
We call $\vec a(n:b)$ an \emph{$n$-variant of $\vec a$}\index{variant|mainidx}.
\end{definition}

\begin{definition} 
If \(V \subseteq\, \^NA\) is a team and \(b \in A\), define 
\[
V(n:b) = \setof{\vec a(n:b)}{\vec a \in V}.\index{$V(n:b)$|mainidx}
\]

Furthermore, if \(B \subseteq A\) define
\[
V(n:B) = \setof{\vec a(n:b)}{\vec a \in V,\ b \in B}. \index{$V(n:B)$|mainidx}
\]
A set \(V' \subseteq V(n:A)\) is called an \emph{$n$-variation of $V$} if for every \(\vec a \in V\) there is at least one $n$-variant of $\vec a$ in $V'$.
Finally if \(f\colon V \to A\), and \(V' \subseteq V\), define the \emph{$n$-variation\index{variation|mainidx} of $V'$ by $f$} to be
\[
V'(n:f) = \setof{\vec a(n:f(\vec a))}{\vec a \in V'}.\index{$V(n:f)$|mainidx}
\]
\end{definition}

\begin{theorem}[Hodges, \textit{cf.} Theorem 1.32 in \cite{Mann:2007}] 
\label{trump semantics}
\index{$\A \modeltf \phi[V]$|mainidx}
Let $\phi$ be an IFG$_N$-formula, let $\A$ be a suitable structure, and let \(V,W \subseteq \^NA\).
\begin{itemize}
	\item If $\phi$ is atomic, then  
	\begin{itemize}
		\item[(+)] \(\A \modelt \phi[V]\) if and only if for every \(\vec a \in V\), \(\A \models \phi[\vec a]\),
		\item[($-$)] \(\A \modelf \phi[W]\) if and only if for every \(\vec b \in W\), \(\A \not\models \phi[\vec b]\).	
	\end{itemize}
	\item If $\phi$ is ${\hneg\psi}$, then 
	\begin{itemize}
		\item[(+)] \(\A \modelt {\hneg\psi[V]}\) if and only if \(\A \modelf \psi[V]\),
		\item[($-$)] \(\A \modelf {\hneg\psi[W]}\) if and only if \(\A \modelt \psi[W]\).
	\end{itemize}
	\item If $\phi$ is $\psi_1 \hor{J} \psi_2$, then 
	\begin{itemize}
		\item[(+)] \(\A \modelt \psi_1 \hor{J} \psi_2[V]\) if and only if \(\A \modelt \psi_1[V_1]\) and \( \A \modelt \psi_2[V_2]\) for some \(V = V_1 \cup_J V_2\),
		\item[($-$)] \(\A \modelf \psi_1 \hor{J} \psi_2[W]\) if and only if \(\A \modelf \psi_1[W]\) and \( \A \modelf \psi_2[W]\).
	\end{itemize}
	\item If $\phi$ is $\hexists{v_n}{J}\psi$, then 
	\begin{itemize}
		\item[(+)] \(\A \modelt \hexists{v_n}{J}\psi[V]\) if and only if \(\A \modelt \psi[V(n:f)]\) for some \(f\colon V \toind{J} A\),
		\item[($-$)] \(\A \modelf \hexists{v_n}{J}\psi[W]\) if and only if \(\A \modelf \psi[W(n:A)]\).
	\end{itemize}
\end{itemize}
\end{theorem}

Recall that the universe of the $N$-dimensional cylindric set algebra over $\A$, denoted $\Cyls_{N}(\A)$, consists of the meanings of all the $N$-variable, first-order formulas expressible in the language of $\A$, where the meaning of a formula is defined by 
\[
\phi^\A = \setof{\vec a \in \^NA}{\A \models \phi[\vec a]}.
\]
Similarly, the universe of the IFG$_N$-cylindric set algebra over $\A$, denoted \(\Cyls_{\mathrm{IFG}_{N}}(\A)\), consists of the meanings of all the IFG$_N$-formulas expressible in the language of $\A$, where the meaning of an IFG$_N$-formula is given by
\[
\trump{\phi}_\A = \setof{V \subseteq \^NA}{\A \modelt \phi[V]}, 
\qquad
\cotrump{\phi}_\A = \setof{W \subseteq \^NA}{\A \modelf \phi[W]},
\]
\[
\norm{\phi}_\A = \tuple{\trump{\phi}_\A, \cotrump{\phi}_\A}.
\]

More generally, we can define IFG$_N$-cylindric set algebras  without reference to a base structure $\A$.

\begin{definition} 
An \emph{IFG-cylindric power set algebra}\index{independence-friendly cylindric power set algebra|mainidx} is an algebra whose universe is \(\powerset(\powerset(\^NA)) \times \powerset(\powerset(\^NA))\), where $A$ is a set and $N$ is a natural number. The set $A$ is called the \emph{base set}\index{base set|mainidx}, and the number $N$ is called the \emph{dimension}\index{dimension|mainidx} of the algebra. Since each  element $X$ is an ordered pair, we will use the notation $X^+$\index{$X^+$ \quad truth coordinate of $X$|mainidx} to refer to the first coordinate of the pair, and $X^-$\index{$X^-$ \quad falsity coordinate of $X$|mainidx} to refer to the second coordinate. There are a finite number of operations: 

\begin{itemize}
	\item the constant \(0 = \tuple{\set{\emptyset}, \powerset(\^NA)}\);
	\item the constant \(1 = \tuple{\powerset(\^NA), \set{\emptyset}}\);
	\item for all \(i,j < N\), the constant $D_{ij}$\index{$D_{ij}$ \quad diagonal element|mainidx} is defined by 
	\begin{itemize}
		\item[(+)] \(D_{ij}^+ = \powerset(\setof{\vec a \in\, \^NA}{a_i = a_j})\),
		\item[($-$)] \(D_{ij}^- = \powerset(\setof{\vec a \in\, \^NA}{a_i \not= a_j})\);
	\end{itemize}
	\item if \(X = \tuple{X^+,\, X^-}\), then \(\n{X}\index{$\n{X}$ \quad negation of $X$|mainidx} = \tuple{X^-, X^+}\);
	\item for every \(J \subseteq N\), the binary operation $+_J$\index{$X +_J Y$|mainidx} is defined by
	\begin{itemize}
		\item[(+)] \(V \in (X +_J Y)^+\) if and only if \(V = V_1 \cup_J V_2\) for some
				 \(V_1 \in X^+\) and \(V_2 \in Y^+\),
		\item[($-$)] \((X +_J Y)^- = X^- \cap Y^-\);
	\end{itemize}
	\item for every \(J \subseteq N\), the binary operation $\cdot_J$\index{$X \cdot_J Y$|mainidx} is defined by
	\begin{itemize}
		\item[(+)] \((X \cdot_J Y)^+ = X^+ \cap Y^+\),
		\item[($-$)] \(W \in (X \cdot_J Y)^-\) if and only if \(W = W_1 \cup_J W_2\) for some
				 \(W_1 \in X^-\) and \(W_2 \in Y^-\);
	\end{itemize}
	\item for every \(n < N\) and \(J \subseteq N\), the unary operation $C_{n,J}$\index{$C_{n,J}(X)$ \quad cylindrification of $X$|mainidx} is defined by
		\begin{itemize}
			\item[(+)] \(V \in C_{n,J}(X)^+\) if and only if \(V(n:f) \in X^+\) for some 
					\(f\colon V \toind{J} A\),
			\item[($-$)] \(W \in C_{n,J}(X)^- \) if and only if \(W(n:A) \in X^-\).
		\end{itemize}
\end{itemize}
\end{definition}

\begin{definition} 
An \emph{IFG-cylindric set algebra}\index{independence-friendly cylindric set algebra|mainidx} (or \emph{IFG-algebra}, for short) is any subalgebra of an IFG-cylindric power set algebra. An \emph{IFG$_N$-cylindric set algebra}\index{IFG$_N$-cylindric set algebra|mainidx} (or \emph{IFG$_N$-algebra}) is an IFG-cylindric set algebra of dimension $N$.
\end{definition}
\pagebreak[3]

%% file: 2-Perfect_IFG_formulas.tex

\section{Perfect IFG-formulas}

In \cite{Hodges:1997b}, Hodges observes that ordinary first-order formulas have the property that \(\A \modelt \phi[V]\) if and only if \(\A \models \phi[\vec a]\) for every \(\vec a \in V\). Independence-friendly formulas with the same property he calls \emph{flat}. In fact, slightly more is true.

\begin{theorem}[Theorem 1.7 in \cite{Mann:2007}] 
\label{conservative extension}
Let $\phi$ be a first-order formula with $N$ variables. We can treat $\phi$ as an IFG$_N$-formula if we interpret $\neg$ as $\hneg$, $\lor$ as $\hor{\emptyset}$, and $\exists v_n$ as $\hexists{v_n}{\emptyset}$. If we do so, then for every suitable structure $\A$ and \(V,W \subseteq \^NA\), 
\begin{enumerate}
	\item  \(\A \modelt \phi[V]\) if and only if \(\A \models \phi[\vec a]\) for all \(\vec a \in V\),
	\item  \(\A \modelf \phi[W]\) if and only if \(\A \not\models \phi[\vec b]\) for all \(\vec b \in W\).
\end{enumerate}
\end{theorem}

Hence, for a first-order sentence, \(\A \modelt \phi\) if and only if \(\A \models \phi\), and \(\A \modelf \phi\) if and only if \(\A \not\models \phi\). Thus IFG-logic is a conservative extension of ordinary first-order logic in the sense that every ordinary first-order formula has a corresponding IFG-formula that is true and false in exactly the same models. The IFG-formulas that correspond to ordinary first-order formulas are exactly those whose independence sets are empty, making the semantic game a game of perfect information.

\begin{definition} 
An IFG-formula $\phi$ is \emph{perfect}\index{IFG-formula!perfect|mainidx} if all of its independence sets are empty. Every perfect IFG-formula is equivalent to the ordinary first-order formula obtained by omitting the empty subscripts.
\end{definition}

\begin{definition}
Given any IFG-formula $\phi$, the \emph{perfection of $\phi$}, denoted $\phi_\emptyset$\index{$\phi_\emptyset$ \quad perfection of $\phi$|mainidx}, is defined recursively as follows.
\begin{itemize}
	\item If $\phi$ is atomic, then $\phi_\emptyset$ is $\phi$.
	\item $(\hneg\psi)_\emptyset$ is $\hneg(\psi_\emptyset)$.
	\item $(\psi_1 \hor{J} \psi_2)_\emptyset$ is $(\psi_1)_\emptyset \hor{\emptyset} (\psi_2)_\emptyset$.
	\item $(\hexists{v_n}{J}\psi)_\emptyset$ is $\hexists{v_n}{\emptyset}\psi$.
\end{itemize}
Thus $\phi_\emptyset$ is just $\phi$ with all of the independence sets changed to $\emptyset$.
\end{definition}


An important feature of the perfection process is that no winning strategies are lost. 

\begin{proposition}
\label{perfection preserves strategies}
If \(\A \modeltf \phi[V]\), then \(\A \modeltf \phi_\emptyset[V]\).
\end{proposition}

\begin{proof}
If $\phi$ is atomic, then $\phi$ is $\phi_\emptyset$. If $\phi$ is ${\hneg\psi}$, then \(\A \modeltf {\hneg\psi[V]}\) if and only if \(\A \modelft \psi[V]\), which implies (by inductive hypothesis) \(\A \modelft \psi_\emptyset[V]\), which holds if and only if \(\A \modeltf {\hneg(\psi_\emptyset)[V]}\).

Suppose $\phi$ is \(\psi_1 \hor{J} \psi_2\). If \(\A \modelt \psi_1 \hor{J} \psi_2[V]\), there is a disjoint cover \(V = V_1 \cup V_2\) such that \(\A \modelt \psi_1[V_1]\) and \(\A \modelt \psi_2[V_2]\). By inductive hypothesis, \(\A \modelt (\psi_1)_\emptyset[V_1]\) and \(\A \modelt (\psi_2)_\emptyset[V_2]\). Hence \(\A \modelt (\psi_1)_\emptyset \hor{\emptyset} (\psi_2)_\emptyset[V]\), which is the same as \(\A \modelt (\psi_1 \hor{J} \psi_2)_\emptyset[V]\). If \(\A \modelf \psi_1 \hor{J} \psi_2[V]\), then \(\A \modelf \psi_1[V]\) and \(\A \modelf \psi_2[V]\). By inductive hypothesis, \(\A \modelf (\psi_1)_\emptyset[V]\) and \(\A \modelf (\psi_2)_\emptyset[V]\). Hence \(\A \modelf (\psi_1)_\emptyset \hor{\emptyset} (\psi_2)_\emptyset[V]\), which is the same as \(\A \modelf (\psi_1 \hor{J} \psi_2)_\emptyset[V]\).

Suppose $\phi$ is \(\hexists{v_n}{J}\psi\). If \(\A \modelt \hexists{v_n}{J}\psi[V]\), then \(\A \modelt \psi[V(n:f)]\) for some function \(f\colon V \to A\). By inductive hypothesis, \(\A \modelt \psi_\emptyset[V(n:f)]\). Hence \(\A \modelt \hexists{v_n}{\emptyset}(\psi_\emptyset)[V]\), which is the same as \(\A \modelt (\hexists{v_n}{J}\psi)_\emptyset[V]\). If \(\A \modelf \hexists{v_n}{J}\psi[V]\), then \(\A \modelf \psi[V(n:A)]\). By inductive hypothesis, \(\A \modelf \psi_\emptyset[V(n:A)]\). Hence \(\A \modelf \hexists{v_n}{\emptyset}(\psi_\emptyset)[V]\), which is the same as \(\A \modelf (\hexists{v_n}{J}\psi)_\emptyset[V]\).
\end{proof}

%% file: 3-Embedding_CsA.tex

\section{Embedding $\Cyls_N(\A)$ into $\Cyls_{\mathrm{IFG}_N}(\A)$}

Meanings of IFG-formulas have the property that \(\trump{\phi} \cap\ \cotrump{\phi} = \set{\emptyset}\), and \(V' \subseteq V \in \norm{\phi}^\pm\) implies \(V' \in \norm{\phi}^\pm\). This fact inspires the following definitions.

\begin{definition} 
A nonempty set $X^* \subseteq\powerset(\^NA)$ is called a \emph{suit}\index{suit|mainidx} if \(V' \subseteq V \in X^*\) implies \(V' \in X^*\). A \emph{double suit}\index{double suit|mainidx} is a pair $\tuple{X^+, X^-}$ of suits such that \(X^+ \cap X^- = \set{\emptyset}\).
\end{definition}

\begin{definition} 
An IFG-algebra is \emph{suited}\index{IFG-cylindric set algebra!suited|mainidx} if all of its elements are pairs of suits. It is \emph{double-suited}\index{IFG-cylindric set algebra!double-suited|mainidx} if all of its elements are double suits.
\end{definition}


\begin{proposition}[Proposition 2.10 in \cite{Mann:2008}] 
\label{subalgebra generated suited}
The subalgebra of an IFG-algebra generated by a set of pairs of suits is a suited IFG-algebra.
\end{proposition}

\begin{proposition}[Proposition 2.11 in \cite{Mann:2008}] 
\label{subalgebra generated double-suited}
The subalgebra of an IFG-algebra generated by a set of double suits is a double-suited IFG-algebra. In particular, \(\Cyls_{\mathrm{IFG}_{N}}(\A)\) is a double-suited IFG$_N$-algebra.
\end{proposition}

Given a set $A$, let $\Suit_N(A)$\index{$\Suit_N(A)$ \quad suited IFG$_N$-cylindric set algebra over $A$} denote the IFG$_N$-algebra whose universe is the set of all pairs of suits in \(\powerset(\powerset(\^NA)) \times \powerset(\powerset(\^NA))\). Let $\DSuit_N(A)$\index{$\DSuit_N(A)$ \quad double-suited IFG$_N$-cylindric set algebra over $A$} denote the IFG$_N$-algebra whose universe is the set of all double suits in \(\powerset(\powerset(\^NA)) \times \powerset(\powerset(\^NA))\). Thus \(\Cyls_{\mathrm{IFG}_{N}}(\A) \subseteq \DSuit_N(A)\).



\begin{definition} 
A double suit $X$ is \emph{flat}\index{double suit!flat|mainidx} if there is a \(V \subseteq \^NA\) such that \(X^+ = \powerset(V)\).
\end{definition}

\begin{definition} 
A double suit $X$ is \emph{perfect}\index{double suit!perfect|mainidx} if there is a \(V \subseteq \^NA\) such that 
\[
X = \tuple{\powerset(V),\, \powerset(\^NA \setminus V)}.
\]
\end{definition}

\begin{proposition} 
\label{perfect double suit}
A double suit $X$ is perfect if and only if \(X +_\emptyset \n{X} = 1\) if and only if \(X \cdot_\emptyset \n{X} = 0\).
\end{proposition}

\begin{proof}
Suppose \(X = \tuple{\powerset(V), \powerset(\^NA \setminus V)}\). Then \(\^NA = V \cup_\emptyset (\^NA \setminus V)\), where \(V \in X^+\) and \(\^NA \setminus V \in (\n{X})^+\). Hence \(\^NA \in (X +_\emptyset \n{X})^+\). Since \(X +_\emptyset \n{X}\) is a double suit we have \(X +_\emptyset \n{X} = \tuple{\powerset(\^NA), \set{\emptyset}} = 1\). 
Conversely, suppose \(X +_\emptyset \n{X} = 1\). Then there exist \(V \in X^+\) and \(V' \in X^-\) such that \(\^NA = V \cup_\emptyset V'\). But then \(V' = \^NA \setminus V\). Since $X$ is a double suit, \(X = \tuple{\powerset(V), \powerset(\^NA \setminus V)}\).
\end{proof}

Since IFG logic is a conservative extension of ordinary first-order logic, we should expect the ordinary cylindric set algebra $\Cyls_N(\A)$\index{$\Cyls_N(\A)$ \quad $N$-dimensional cylindric set algebra over $\A$} to be embeddable into some reduct of $\Cyls_{\mathrm{IFG}_{N}}(\A)$.

\begin{definition} 
The reduct of an IFG$_N$-algebra to the signature $\seq{0, 1, D_{ij}, \n{ }, +_\emptyset, \cdot_\emptyset, C_{n,\emptyset}}$ is called the \emph{$\emptyset$-reduct}\index{IFG-cylindric set algebra!$\emptyset$-reduct|mainidx} of the algebra. A subalgebra of the $\emptyset$-reduct is called a \emph{perfect subalgebra}\index{perfect subalgebra|mainidx} if all of its members are perfect. The subalgebra of the $\emptyset$-reduct of $\Cyls_{\mathrm{IFG}_{N}}(\A)$ generated by the meanings of atomic formulas is denoted $\Cyls_{\mathrm{IFG}_{N, \emptyset}}(\A)$\index{$\Cyls_{\mathrm{IFG}_{N, \emptyset}}(\A)$|mainidx}.
\end{definition}

It is worth noting that $\Cyls_{\mathrm{IFG}_{N}}(\A)$ is generated by the set of its perfect elements because it is generated by the meanings of atomic formulas, which are all perfect.

\begin{lemma} 
Suppose \(X = \tuple{\powerset(V),\, \powerset(\^NA \setminus V)}\) and \(Y = \tuple{\powerset(W),\, \powerset(\^NA \setminus W)}\). Then 
\begin{enumerate}
	\item \(X +_\emptyset Y = \tuple{\powerset(V \cup W),\, \powerset(\^NA \setminus (V \cup W))}\), 
	\item \(C_{n,\emptyset}(X) = \tuple{\powerset(V(n:A)),\, \powerset(\^NA \setminus V(n:A))}\).
\end{enumerate}
\end{lemma}

\begin{proof}
(a)
Suppose \(U \in (X +_\emptyset Y)^+\). Then \(U = U_1 \cup_\emptyset U_2\) for some \(U_1 \in \powerset(V)\) and \(U_2 \in \powerset(W)\). Hence \(U \in \powerset(V \cup W)\). Conversely, suppose \(U \in \powerset(V \cup W)\). Let \(U_1 = U \cap V\) and \(U_2 = U \setminus V\). Then \(U = U_1 \cup_\emptyset U_2\) where \(U_1 \in \powerset(V)\) and \(U_2 \in \powerset(W)\). Hence \(U \in (X +_\emptyset Y)^+\). Also observe 
\begin{align*}
(X +_\emptyset Y)^- 
&= \powerset(\^NA \setminus V) \cap \powerset(\^NA \setminus W) \\
&= \powerset((\^NA \setminus V) \cap (\^NA \setminus W)) \\
&= \powerset(\^NA \setminus (V \cup W)).
\end{align*}

(b)
Suppose \(U \in C_{n,\emptyset}(X)^+\). Then \(U(n:f) \in \powerset(V)\) for some \(f\colon U \toind{\emptyset} A\), so \(U \subseteq V(n:A)\). Hence \(U \in \powerset(V(n:A))\). Conversely, suppose \(U \in \powerset(V(n:A))\). Then \(U \subseteq V(n:A)\), which means that for every \(\vec a \in U\) there is a \(\vec b \in V\) such that \(\vec a = \vec b(n: a_n)\). Let \(f\colon U \to V\) be a function that chooses one such $\vec b$ for every $\vec a$. Then \(\pr_n \circ f\colon U \toind{\emptyset} A\) and \(U(n: \pr_n\circ f) \in \powerset(V)\). Hence \(U \in C_{n,\emptyset}(X)^+\). Also observe that \(U \in C_{n,\emptyset}(X)^-\) 
if and only if \(U(n:A) \in \powerset(\^NA \setminus V)\) 
if and only if \(U(n:A) \subseteq \^NA \setminus V\) 
if and only if \(U \subseteq \^NA \setminus V(n:A)\) 
if and only if \(U \in \powerset(\^NA \setminus V(n:A))\).
\end{proof}

\begin{proposition} 
\label{subalgebra generated perfect}
Let $\C$ be the $\emptyset$-reduct of an IFG$_N$-algebra. Every subalgebra of $\C$ generated by a set of perfect elements is perfect.
\end{proposition}

\begin{proof}
The constants $0$, $1$, and $D_{ij}$ are all perfect. If \(X = \tuple{\powerset(V),\, \powerset(\^NA \setminus V)}\) is perfect, then so is \(\n{X} = \tuple{\powerset(\^NA \setminus V),\, \powerset(V)}\). By the previous lemma, if $X$ and $Y$ are perfect so are $X +_\emptyset Y$ and $C_{n,\emptyset}(X)$.
\end{proof}

\begin{corollary} 
\label{perfect subalgebra}
\(\Cyls_{\mathrm{IFG}_{N, \emptyset}}(\A)\) is perfect.
\end{corollary}

\begin{proof}
\(\Cyls_{\mathrm{IFG}_{N, \emptyset}}(\A)\) is generated by the meanings of atomic formulas, which are all perfect.
\end{proof}

It follows that if $\phi$ is a perfect IFG-formula, then $\norm{\phi}$ is perfect. It is conceivable that $\Cyls_{\mathrm{IFG}_N}(\A)$ includes perfect elements that cannot be generated by $\emptyset$-operations from the meanings of atomic formulas. The next proposition shows that this is in fact not the case.

\begin{proposition} 
\label{no rogue perfect elements}
Every perfect element in $\Cyls_{\mathrm{IFG}_{N}}(\A)$ belongs to $\Cyls_{\mathrm{IFG}_{N, \emptyset}}(\A)$.
\end{proposition}

\begin{proof}
Let \(X = \norm{\phi} = \tuple{\powerset(V), \powerset(\^NA \setminus V)}\), and consider \(\norm{\phi_\emptyset}\). We know $\norm{\phi_\emptyset}$ is perfect, so let \(\norm{\phi_\emptyset}= \tuple{\powerset(V_\emptyset), \powerset(\^NA \setminus V_\emptyset)}\). It suffices to show that \(V = V_\emptyset\). By \propref{perfection preserves strategies}, \(\A \modelt \phi[V]\) implies \(\A \modelt \phi_\emptyset[V]\), so \(V \subseteq V_\emptyset\). Conversely, \(\A \modelf \phi[\^NA \setminus V]\) implies \(\A \modelf \phi_\emptyset[\^NA \setminus V]\), so \(\^NA \setminus V \subseteq \^NA \setminus V_\emptyset\). Hence \(V = V_\emptyset\).
\end{proof}

\begin{theorem} 
\label{Cs(A) isomorphic}
\(\Cyls_N(\A) \iso \Cyls_{\mathrm{IFG}_{N,\emptyset}}(\A)\).
\end{theorem}

\begin{proof}
Define two functions \(F\colon \Cyls_N(\A) \to \Cyls_{\mathrm{IFG}_{N, \emptyset}}(\A)\) and \(G\colon \Cyls_{\mathrm{IFG}_{N, \emptyset}}(\A) \to \Cyls_N(\A)\) by 
\[
F(V) = \tuple{\powerset(V),\, \powerset(\^NA \setminus V)} \quad \text{and} \quad G(X) = \bigcup X^+,
\]
respectively.
Observe that 
\[
G \circ F(V) = \bigcup \powerset(V) = V,
\]
\[
F \circ G(X) = \tuple{\powerset\left(\bigcup X^+\right), \powerset\left(\^NA \setminus \bigcup X^+\right)} = X.
\]
Thus $F$ is bijective.

To show that $F$ is a homomorphism, observe that 
\begin{align*}
F(0) &= F(\emptyset) = \tuple{\powerset(\emptyset), \powerset(\^NA)} 
= 0, \\
F(1) &= F(\^NA) = \tuple{\powerset(\^NA), \powerset(\emptyset)} 
= 1, \\
F(D_{ij}) &= F(\setof{\vec a \in \^NA}{a_i = a_j}) \\
	 &= \tuple{\powerset(\setof{\vec a \in \^NA}{a_i = a_j}),\, \powerset(\setof{\vec a \in \^NA}{a_i \not= a_j})} \\
	 &= D_{ij}, \\
F(-V) &= F(\^NA \setminus V) = \tuple{\powerset(\^NA \setminus V),\, \powerset(V)} = \n{F(V)}, \\
F(V + W) &= F(V \cup W) = \tuple{\powerset(V \cup W),\, \powerset(\^NA \setminus (V \cup W)} = F(V) +_\emptyset F(W), \\
F(C_n(V)) &= F(V(n:A)) = \tuple{\powerset(V(n:A)),\, \powerset(\^NA \setminus V(n:A))} = C_{n,\emptyset}(F(V)).
\end{align*}
Therefore $F$ is an isomorphism.
\end{proof}

%% file: 4-Conclusion.tex

\section{Conclusion}

An IFG-formula $\phi$ has the same meaning in $\A$ as an ordinary first-order formula if and only if $\norm{\phi}_\A$ is perfect. The ordinary cylindric set algebra $\Cyls_N(\A)$ is isomorphic to the subalgebra of the $\emptyset$-reduct of $\Cyls_{\mathrm{IFG}_{N}}(\A)$ consisting of all perfect elements.

%% file: perfect_IFG-formulas.bbl
\begin{thebibliography}{1}

\bibitem{Dechesne:2005}
Francien Dechesne.
\newblock {\em Game, Set, Maths: Formal investigations into logic with
  imperfect information}.
\newblock PhD thesis, Universiteit van Tilburg, March 2005.

\bibitem{Hintikka:1996}
Jaakko Hintikka.
\newblock {\em The Principles of Mathematics Revisited}.
\newblock Cambridge University Press, 1996.

\bibitem{Hintikka:1989}
Jaakko Hintikka and Gabriel Sandu.
\newblock Informational independence as a semantical phenomenon.
\newblock In Jens~Erik Fenstad et~al., editors, {\em Logic, Methodology and
  Philosophy of Science VIII}, volume 126 of {\em Studies in Logic and the
  Foundations of Mathematics}, pages 571--589. North-Holland, 1989.

\bibitem{Hodges:1997a}
Wilfrid Hodges.
\newblock Compositional semantics for a language of imperfect information.
\newblock {\em Logic Journal of the IGPL}, 5(4):539--563, 1997.

\bibitem{Hodges:1997b}
Wilfrid Hodges.
\newblock Some strange quantifiers.
\newblock In Jan Mycielski, Grzegorz Rozenberg, and Arto Salomaa, editors, {\em
  Structures in Logic and Computer Science: A Selection of Essays in Honor of
  A. Ehrenfeucht}, number 1261 in Lecture Notes in Computer Science, pages
  51--65. Springer, 1997.

\bibitem{Mann:2008}
Allen~L. Mann.
\newblock Independence-friendly cylindric set algebras.
\newblock arXiv:0711.4376.

\bibitem{Mann:2007}
Allen~L. Mann.
\newblock {\em Independence-Friendly Cylindric Set Algebras}.
\newblock PhD thesis, University of Colorado at Boulder, 2007.

\end{thebibliography}
